\documentclass[leqno,twoside]{article}
\usepackage{amssymb, amsmath, latexsym, amsthm, enumerate, amsfonts}
\title{} \author{} \date{}
\usepackage[pdftex, pdfstartview=FitH]{hyperref}

\newtheorem{thm}{Theorem}[section]

\newtheorem{lem}[thm]{Lemma}
\newtheorem{prop}[thm]{Proposition}
\theoremstyle{definition}
\newtheorem{defn}[thm]{Definition}

\theoremstyle{remark}
\newtheorem{rem}[thm]{Remark}

\numberwithin{equation}{section}

\usepackage[paperwidth=165mm, paperheight=235mm, twoside, hmargin={25mm,20mm}, vmargin={20mm,20mm}]{geometry}

\begin{document}

\thispagestyle{empty}

\begin{center}
{\large \bf FINSLER STRUCTURES ON HOLOMORPHIC LIE ALGEBROIDS}
\vspace*{3mm}

{\bf Alexandru Ionescu\footnote{Faculty of Mathematics and Informatics, \emph{Transilvania} University of Bra\c sov,  Romania,\\ e-mail: \href{mailto:alexandru.codrin.ionescu@gmail.com}{alexandru.codrin.ionescu@gmail.com}}}
\end{center}

\begin{abstract}
Complex Finsler vector bundles have been studied mainly by T. Aikou, who defined complex Finsler structures on holomorphic vector bundles. In this paper, we consider the more general case of a holomorphic Lie algebroid and we introduce Finsler structures, partial and Chern-Finsler connections on it.

First, we recall some basic notions on holomorphic Lie algebroids. Then, using an idea from E. Martinez, we introduce the concept of complexified prolongation of such an algebroid. Also, we study nonlinear and linear connections on the tangent bundle $T_{\mathbb{C}}E$ and on the prolongation $\mathcal{T}_{\mathbb{C}}E$ and we investigate the relation between their coefficients. The analogue of the classical Chern-Finsler connection is defined and studied in the paper for the case of the holomorphic Lie algebroid $E$.
\\[2mm] {\it AMS Mathematics  Subject Classification $(2010)$:} 17B66, 53B35, 53B40.
\\[1mm] {\it Key words and phrases:} Holomorphic Lie algebroid, anchor map, Finsler, nonlinear connection, prolongation, Lagrangian structures.
\end{abstract}

\section*{Introduction}

Complex Finsler structures on holomorphic vector bundles have been introduced and studied mainly by T. Aikou (\cite{3,4,5}). In this paper, we introduce Finsler structures, partial and Chern-Finsler connections in the more general case of a holomorphic Lie algebroid. The setting for the geometrical constructions will be the tangent bundle of the algebroid and the prolongation of the algebroid, a concept introduced by E. Martinez (\cite{13,14}) and studied also by L. Popescu \cite{17,18} and E. Peyghan (\cite{16}).

We briefly recall here some general notions and set our notations for the geometry of holomorphic Lie algebroids. More ideas can be found in \cite{8,9}.

Let $M$ be a complex $n$-dimensional manifold and $E$ a holomorphic vector bundle of rank $m$ over $M$. Denote by $\pi:E\rightarrow M$ the holomorphic bundle projection, by $\Gamma(E)$ the module of holomorphic sections of $\pi$ and let $T_{\mathbb{C}}M = T'M\oplus T''M$ be the complexified tangent bundle of $M$, split into the holomorphic and antiholomorphic tangent bundles.

The holomorphic vector bundle $E$ over $M$ is called anchored if there exists a holomorphic vector bundle morphism $\rho: E\rightarrow T'M$, called anchor map.

The following definition is known from \cite{6,7,10,12,20}.

A \textit{holomorphic Lie algebroid} over $M$ is a triple $(E,[\cdot,\cdot]_E,\rho_E)$, where $E$ is a holomorphic vector bundle anchored over $M$, $[\cdot,\cdot]_E$ is a Lie bracket on $\Gamma(E)$ and $\rho_E:\Gamma(E)\rightarrow\Gamma(T'M)$ is the homomorphism of complex modules induced by the anchor map $\rho$ such that 
\begin{equation} \label{1}
[s_1,fs_2]_E = f[s_1,s_2]_E+\rho_E(s_1)(f)s_2
\end{equation}
for all $s_1,s_2\in\Gamma(E)$ and all $f\in Hol(M)$.

As a consequence of this definition, we have that $\rho_E([s_1,s_2]_E) = [\rho_E(s_1),\\ \rho_E(s_2)]_{T'M}$ (\cite{12}), which means that $\rho_E:(\Gamma(E),[\cdot,\cdot]_E)\rightarrow(\Gamma(T'M),[\cdot,\cdot])$ is a complex Lie algebra homomorphism.

Locally, if $\{z^{k}\}_{k=\overline{1,n}}$ is a complex coordinate system on $U\subset M$ and $\{e_{\alpha }\}_{\alpha =\overline{1,m}}$ is a local frame of holomorphic sections of $E$ on $U$, then $(z^k,u^\alpha)$ are local complex coordinates on $\pi^{-1}(U)\subset E$, where $u = u^\alpha e_\alpha(z)\in E$.

The action of the holomorphic anchor map $\rho_E$ can locally be described by
\begin{equation} \label{5}
\rho_E(e_\alpha) = \rho_\alpha^k\dfrac{\partial}{\partial z^k},
\end{equation}
while the Lie bracket $[\cdot,\cdot]_E$ is locally given by 
\begin{equation} \label{6}
[e_\alpha,e_\beta]_E = \mathcal{C}^{\:\gamma}_{\alpha\beta}e_\gamma.
\end{equation}
The holomorphic functions $\rho^k_\alpha = \rho^k_\alpha(z)$ and $\mathcal{C}^{\:\gamma}_{\alpha\beta} = \mathcal{C}^{\:\gamma}_{\alpha\beta}(z)$ on $M$ are called \emph{the holomorphic anchor coefficients} and \emph{the holomorphic structure functions} of the Lie algebroid E, respectively.

Since $E$ is a holomorphic vector bundle, the natural complex structure acts on its sections by $J_E(e_\alpha) = ie_\alpha$ and $J_E(\bar{e}_\alpha) = -i\bar{e}_\alpha$. Hence, the complexified bundle $E_{\mathbb{C}}$ of $E$ decomposes into $E_{\mathbb{C}} = E'\oplus E''$. The sections of $E_{\mathbb{C}}$ are given as usual by $\Gamma(E') = \{s-iJ_Es\:|\:s\in\Gamma(E)\}$ and $\Gamma(E'') = \{s+iJ_Es\:|\:s\in\Gamma(E)\}$, respectively. The local basis of sections of $E'$ is $\{e_\alpha\}_{\alpha=\overline{1,m}}$, while for $E''$, the basis is represented by their conjugates $\{\bar{e}_\alpha := e_{\bar{\alpha}}\}_{\alpha=\overline{1,m}}$. Since $\rho_E:\Gamma(E)\rightarrow\Gamma(T'M)$ is a homomorphism of complex modules, it extends naturally to the complexified bundle by $\rho'(e_\alpha) = \rho_E(e_\alpha)$ and $\rho''(e_{\bar{\alpha}}) = \rho_E(e_{\bar{\alpha}})$. Thus, we can write $\rho_E = \rho'\oplus \rho''$ on the complexified bundle, and since $E$ is holomorphic, the functions $\rho(z)$ are holomorphic, hence $\rho_\alpha^{\bar{k}} = \rho_{\bar{\alpha}}^k = 0$ and $\rho_{\bar{\alpha}}^{\bar{k}} = \overline{\rho_\alpha^k}$.

As a vector bundle, the holomorphic Lie algebroid $E$ has a natural structure of complex manifold. If we wish to study a Finsler structure on the manifold $E$ with coordinates in a local chart $(z^k,u^\alpha)$, is of interest to consider the action of such a structure on the sections of the complexified tangent bundle $T_{\mathbb{C}}E$.

Two approaches on the tangent bundle of a holomorphic Lie algebroid $E$ were described in \cite{9}. The first is the classical study of the tangent bundle of $E$, while the second is that of the prolongation of $E$. The latter idea is due to E. Martinez \cite{13,14} based on the work of P. Liberman  \cite{11} and it appeared from the need of a formalism for ordinary Lagrangian Mechanics where equations of motion could be obtained independent of structures on the dual of the algebroid, a problem raised by Weinstein (\cite{19,20}). For introducing geometrical objects such as nonlinear connections or sprays which could be studied in a similar manner to the tangent bundle of a complex manifold, this setting seems therefore more attractive for studying the Finsler structure. Further, we briefly introduce the notion of prolongation of a holomorphic Lie algebroid and complete the study with some results on its complexification.

\section{The prolongation of a holomorphic Lie algebroid}

\label{S1}

Let us recall from \cite{9} some ideas on the prolongation of the holomorphic Lie algebroid.

For the holomorphic Lie algebroid $E$ over a complex manifold $M$, its prolongation will be introduced using the tangent mapping $\pi'_\ast: T'E\rightarrow T'M$ and the holomorphic anchor map $\rho_E: E\rightarrow T'M$. Define the subset $\mathcal{T}'E$ of $E\times T'E$ by $\mathcal{T}'E = \{(e,v)\in
E\times T'E\ |\ \rho(e) = \pi'_\ast(v)\}$ and the mapping $\pi'_{\mathcal{T}}: \mathcal{T}'E\rightarrow E$, given by $\pi'_{\mathcal{T}}(e,v) = \pi_E(v)$, where $\pi'_E: T'E\rightarrow E$ is the tangent projection. Then $(\mathcal{T}'E,\pi'_{\mathcal{T}},E)$ is a holomorphic vector bundle over $E$, of rank $2m$. Moreover, it is easy to verify that the projection onto the second factor $\rho'_{\mathcal{T}}: \mathcal{T}'E\rightarrow T'E$, $\rho'_{\mathcal{T}}(e,v) = v$, is the
anchor of a new holomorphic Lie algebroid over the complex manifold $E$ (see \cite{13,14,17,18} for details in the real case).

The holomorphic Lie algebroid $E$ has a structure of holomorphic vector bundle with respect to the complex structure $J_E$. Let $E_{\mathbb{C}}$ be the complexified bundle of $E$ and $T_{\mathbb{C}}E = T'E\oplus T''E$, its complexified tangent bundle. A similar idea to that of Martinez (\cite{13,14}) leads to the definition of the complexified prolongation $\mathcal{T}_{\mathbb{C}}E$ of $E$ as follows. We extend $\mathbb{C}$-linearly the tangent mapping $\pi'_\ast: T'E\rightarrow T'M$ and the anchor $\rho_E: E\rightarrow T'M$ to obtain $\pi_{\ast,\mathbb{C}}: T_{\mathbb{C}}E\rightarrow T_{\mathbb{C}}M$ and $\rho_{E,\mathbb{C}}: E_{\mathbb{C}}\rightarrow T_{\mathbb{C}}M$, respectively. If $\pi_{E,\mathbb{C}}: T_{\mathbb{C}}E\rightarrow E_{\mathbb{C}}$ is the tangent projection extended to the complexified spaces, then we can define the subset $\mathcal{T}_{\mathbb{C}}E$ of $E_{\mathbb{C}}\times T_{\mathbb{C}}E$ by 
\begin{equation*}
\mathcal{T}_{\mathbb{C}}E = \{(e,v)\in E_{\mathbb{C}}\times T_{\mathbb{C}}E\ |\ \rho_{E,\mathbb{C}}(e) = \pi_{\ast,\mathbb{C}}(v)\}
\end{equation*}
and the mapping $\pi_{\mathcal{T},\mathbb{C}}: \mathcal{T}_{\mathbb{C}}E\rightarrow E_{\mathbb{C}}$ by $\pi_{\mathcal{T},\mathbb{C}}(e,v) = \pi_{E,\mathbb{C}}(v)$. Thus, we obtain a complex vector bundle $(\mathcal{T}_{\mathbb{C}}E,\pi_{\mathcal{T},\mathbb{C}},E_{\mathbb{C}})$ over $E_{\mathbb{C}}$. Also, the projection onto the second factor, 
\begin{equation*}
\rho_{\mathcal{T},\mathbb{C}}: \mathcal{T}_{\mathbb{C}}E \rightarrow T_{\mathbb{C}}E, \quad \rho_{\mathcal{T},\mathbb{C}}(e,v) = v,
\end{equation*}
is the anchor of a complex Lie algebroid over $E_{\mathbb{C}}$, called \emph{the complexified prolongation} of $E$. Indeed, $\mathcal{T}'E = E'\times T'E$ is a holomorphic product bundle and, since $\rho_{E,\mathbb{C}} = \rho'+\rho''$ and $\pi_{\ast ,\mathbb{C}} = \pi'_\ast+\pi''_\ast$ are holomorphic mappings with $\rho_{E,\mathbb{C}}(e) = \pi_{\ast,\mathbb{C}}(v)$, then $\rho'(\bar{e})=\pi'_\ast(\bar{v}) = 0$. We conclude that the complexified prolongation coincides with the complexification of the prolongation $\mathcal{T}'E$ (as a complex manifold), that is $\mathcal{T}_{\mathbb{C}}E = \mathcal{T}'E\oplus \mathcal{T}''E$, where $\mathcal{T}''E = \overline{\mathcal{T}'E} = E''\times T''E$, with the required restrictions $\rho'_E(e) = \pi'_\ast(v)$ and its conjugate.

The vertical subbundle of the complexified prolongation is defined using the projection onto the first factor $\tau_1: \mathcal{T}'E \rightarrow E$, $\tau_1(e,v) = e$, by 
\begin{equation*}
V\mathcal{T}'E = \ker\tau_1 = \{(e,v)\in\mathcal{T}'E\ |\ \tau_1(e,v) = 0\}.
\end{equation*}
From the construction above, it follows that any element of $V\mathcal{T}'E$ has the form $(0,v)\in E\times \mathcal{T}'E$, with $\pi'_\ast(v) = 0$. Then, vertical elements $(0,v)\in V\mathcal{T}'E$ have the property $v\in\ker\pi'_\ast$. By conjugation, we obtain $V\mathcal{T}''E$ and the complexified vertical subbundle of the prolongation $\mathcal{T}_{\mathbb{C}}E$ is $V\mathcal{T}_{\mathbb{C}}E = V\mathcal{T}'E\oplus V\mathcal{T}''E$.

The local coordinates on $\mathcal{T}'E$ are $(z^k,u^\alpha,v^\alpha,w^\alpha)$, obtained as follows. For an element $e = v^\alpha e_\alpha(z)\in E$ and a vector $v$ tangent at $E$ in $u = u^\alpha e_\alpha(z)\in E$, that is, $v\in T'_uE$, the identity $\rho'(e) = \pi'_*(v)$ yields the vector $v$
in the form 
\begin{equation*}
v = \rho^k_\alpha v^\alpha\dfrac{\partial}{\partial z^k} + w^\alpha\dfrac{\partial}{\partial u^\alpha}.
\end{equation*}

The local basis of holomorphic sections in $\Gamma(\mathcal{T}'E)$ is $\{\mathcal{Z}_\alpha,\mathcal{V}_\alpha\}$, defined by 
\begin{equation*}
\mathcal{Z}_\alpha(u) = \bigg(e_\alpha(\pi(u)),\rho^k_\alpha\dfrac{\partial}{\partial z^k}\bigg|_{u}\bigg), \qquad \mathcal{V}_\alpha(u) = \bigg(0,\dfrac{\partial}{\partial u^\alpha}\bigg|_{u}\bigg),
\end{equation*}
where $\bigg\{\dfrac{\partial}{\partial z^k},\dfrac{\partial}{\partial u^\alpha}\bigg\}$ is the natural frame on $T'E$. Therefore, a local basis of sections in $\Gamma(\mathcal{T}_{\mathbb{C}}E)$ is $\{\mathcal{Z}_\alpha,\mathcal{V}_\alpha,\mathcal{Z}_{\bar{\alpha}},\mathcal{V}_{\bar{\alpha}}\}$, where $\mathcal{Z}_{\bar{\alpha}},\mathcal{V}_{\bar{\alpha}}$ are obtained by conjugation, i.e.,
\begin{equation*}
\mathcal{Z}_{\bar{\alpha}}(u) = \bigg(e_{\bar{\alpha}}(\pi(u)),\rho_{\bar{\alpha}}^{\bar{k}}\dfrac{\partial}{\partial\bar{z}^{k}}\bigg|_{u}\bigg), \quad \mathcal{V}_{\bar{\alpha}}(u) = \bigg(0,\dfrac{\partial}{\partial\bar{u}^{\alpha}}\bigg|_{u}\bigg).
\end{equation*}

For a change of local charts on $E$ given by
\begin{equation*}
\widetilde{z}^k = \widetilde{z}^k(z),\qquad \widetilde{u}^\alpha =
M^\alpha_\beta(z)u^\beta,
\end{equation*}
the basis of sections on $E$, $\{e_\alpha\}$, changes by
\begin{equation}\label{sch.c.e}
\widetilde{e}_\alpha = W^\beta_\alpha e_\beta,
\end{equation}
the local coefficients of the anchor map, $\rho^k_\alpha$, change as
\begin{equation} \label{sch.c.ro}
\widetilde{\rho}^k_\alpha = W^\beta_\alpha \rho^h_\beta \dfrac{\partial \widetilde{z}^k}{\partial z^h},
\end{equation}
while the natural frame of fields $\bigg\{\dfrac{\partial }{
\partial z^{k}},\dfrac{\partial }{\partial u^{\alpha }}\bigg\}$ from $T'E$ changes by the rules
\begin{align}
\dfrac{\partial }{\partial z^{h}}& =\dfrac{\partial \widetilde{z}^{k}}{%
\partial z^{h}}\dfrac{\partial }{\partial \widetilde{z}^{k}}+\dfrac{\partial
M_{\beta }^{\alpha }}{\partial z^{h}}u^{\beta }\dfrac{\partial }{\partial 
\widetilde{u}^{\alpha }},  \label{sch.c.T'E} \\
\dfrac{\partial }{\partial u^{\beta }}& =M_{\beta }^{\alpha }\dfrac{\partial 
}{\partial \widetilde{u}^{\alpha }},  \notag
\end{align}
see \cite{9} for more details. Obviously, since $E$ is a complex manifold, all of the above rules can also be conjugated.

The coordinates on $\mathcal{T}'E$ change by the rules
\begin{align*}
\widetilde{z}^k &= \widetilde{z}^k(z),\\
\widetilde{u}^\alpha &= M^\alpha_\beta u^\beta,\\
\widetilde{v}^\alpha &= M^\alpha_\beta v^\beta,\\
\widetilde{w}^\alpha &= M^\alpha_\beta w^\beta + \rho^k_\beta v^\beta\dfrac{\partial M^\alpha_\gamma}{\partial z^k}u^\gamma.
\end{align*}

Using these, we obtain the rules of change for the local basis of sections $\{\mathcal{Z}_{\alpha },\mathcal{V}_{\alpha },\mathcal{Z}_{\bar{\alpha}},\mathcal{V}_{\bar{\alpha}}\}$ from $\Gamma (\mathcal{T}_{\mathbb{C}}E)$:
\begin{align} \label{sch.sect}
\widetilde{\mathcal{Z}}_\beta &= W^\alpha_\beta\left(\mathcal{Z}_\alpha - \rho^h_\alpha\dfrac{\partial M^\gamma_\varepsilon}{\partial z^h}W^\tau_\gamma u^\varepsilon \mathcal{V}_\tau\right), \\
\widetilde{\mathcal{V}}_\beta &= W^\alpha_\beta\mathcal{V}_\alpha\notag
\end{align}
together with their conjugates.

By using a complete lift (which can be defined naturally) to the prolongation $\mathcal{T}^{\prime }E$, the Liouville vector and an almost tangent structure can be defined: 
\begin{equation} \label{s1}
\mathcal{L}=u^{\alpha }\mathcal{V}_{\alpha} 
\end{equation}
and
\begin{equation} \label{act.str.tg.}
T(\mathcal{Z}^{\alpha })=\mathcal{V}^{\alpha },\qquad T(\mathcal{V}^{\alpha
})=0. 
\end{equation}
A section $\mathcal{S}$ of the holomorphic Lie algebroid $\mathcal{T}^{\prime }E$ is called \textit{complex semispray} on $E$ if
\begin{equation*}
T(\mathcal{S})=\mathcal{L}.
\end{equation*}

The local expression of a semispray on $\mathcal{T}^{\prime }E$ is 
\begin{equation*}
\mathcal{S}=u^{\alpha }\mathcal{Z}_{\alpha }-2G^{\alpha }(z,u)\mathcal{V}_{\alpha }.
\end{equation*}

The local coefficients $G^{\alpha }$ of a special complex semispray can be derived from a Lagrange (Finsler) function on $E$, as proved in Theorem 2.1 from \cite{9}.

As further notations, we shall use the well-known abbreviations 
\begin{equation*}
\dfrac{\partial }{\partial z^{k}}:=\partial_{k},\ \dfrac{\partial }{\partial u^{\alpha }}:=\dot{\partial}_{\alpha },\ \dfrac{\partial }{\partial \bar{z}^{k}}:=\partial_{\bar{k}},\ \dfrac{\partial }{\partial \bar{u}^{\alpha }}:=\dot{\partial}_{\bar{\alpha}}.
\end{equation*}

The action of the anchor map $\rho_{\mathcal{T}}$ on $\mathcal{T}E$ is locally described by 
\begin{align}
\label{act.ro.prel}
&\rho_{\mathcal{T}}(\mathcal{Z}_\alpha) = \rho^k_\alpha\partial_k =: \partial_\alpha,\quad \rho_{\mathcal{T}}(\mathcal{V}_\alpha) = \dot{\partial}_\alpha,\\
\notag
&\rho_{\mathcal{T}}(\mathcal{Z}_{\bar{\alpha}}) = \rho^{\bar{z}}_{\bar{\alpha}}\partial_{\bar{k}}=:\partial_{\bar{\alpha}},\quad \rho_{\mathcal{T}}(\mathcal{V}_{\bar{\alpha}}) = \dot{\partial}_{\bar{\alpha}}.
\end{align}

\begin{prop}
The Lie brackets of the basis $\{\mathcal{Z}_\alpha,\mathcal{V}_\alpha,\mathcal{Z}_{\bar{\alpha}},\mathcal{V}_{\bar{\alpha}}\}$ are 
\begin{align*}
&[\mathcal{Z}_\alpha,\mathcal{Z}_\beta]_\mathcal{T} = \mathcal{C}^{\:\gamma}_{\alpha\beta}\mathcal{Z}_\gamma,
&[\mathcal{Z}_\alpha,\mathcal{V}_\beta]_\mathcal{T} = 0,\qquad
&[\mathcal{V}_\alpha,\mathcal{V}_\beta]_\mathcal{T} = 0, \\
&[\mathcal{Z}_\alpha,\mathcal{Z}_{\bar{\beta}}]_\mathcal{T} = 0, 
&[\mathcal{Z}_\alpha,\mathcal{V}_{\bar{\beta}}]_\mathcal{T} = 0,\qquad
&[\mathcal{V}_\alpha,\mathcal{V} _{\bar{\beta}}]_\mathcal{T} = 0
\end{align*}
and the conjugates, for instance $[\mathcal{Z}_{\bar{\alpha}},\mathcal{Z}_{\bar{\beta}}]_\mathcal{T} = \overline{[\mathcal{Z_\alpha},\mathcal{Z}_\beta]_{\mathcal{T}}} = \mathcal{C}^{\:\bar{\gamma}}_{\bar{\alpha}\bar{\beta}}\mathcal{Z}_{\bar{\gamma}}$, etc.
\end{prop}

\subsection{Nonlinear connections on $\mathcal{T}_{\mathbb{C}}E$}

In \cite{9}, we have considered an adapted frame on $\mathcal{T}^{\prime}E$ given by a complex nonlinear connection. Our aim in the following is to define a complex nonlinear connection on the holomorphic prolongation $\mathcal{T}^{\prime }E$ which is induced by a nonlinear connection on $E$.

A \textit{complex nonlinear connection} on $\mathcal{T}^{\prime }E$ is given by a complex vector subbundle $H\mathcal{T}^{\prime }E$ of $\mathcal{T}^{\prime }E$ such that $\mathcal{T}^{\prime }E=H\mathcal{T}^{\prime }E\oplus V\mathcal{T}^{\prime }E$. If $l^{h}$ is the horizontal lift from $\mathcal{T}'E$ to $H\mathcal{T}'E$, then similar considerations as in the real case (\cite{17}) lead to the following local expression of $l^{h}$: 
\begin{equation*}
l^{h}(\mathcal{Z}_{\alpha })=\mathcal{Z}_{\alpha }-N_{\alpha }^{\beta }\mathcal{V}_{\beta },\qquad l^{h}(\mathcal{V}_{\alpha })=0,
\end{equation*}
where $N_{\alpha }^{\beta }=N_{\alpha }^{\beta }(z,u)$ are functions defined on $E$, called the coefficients of the complex nonlinear connection on $\mathcal{T}^{\prime }E$.

Let us consider that on $T'E$ is given a nonlinear connection with coefficients $N_k^\beta$ such that $\{\delta_k,\dot{\partial}_\alpha\}$ is the adapted frame of fields on $T'E$, where
\begin{equation*}
\delta_k = \partial_k - N_k^\beta\dot{\partial}_\beta
\end{equation*}
(see also \cite{8,9}). Denote by $\delta_\alpha = \rho_\alpha^k\delta_k$ and let
\begin{equation} \label{rel.coef.N}
N_\alpha^\beta = \rho_\alpha^k N_k^\beta.
\end{equation}
Then, these are the coefficients of an induced nonlinear connection on the prolongation $\mathcal{T}'E$, as we proved in \cite{9}.

Denote by 
\begin{equation} \label{7.1}
\mathcal{X}_\alpha=\mathcal{Z}_{\alpha }-N_{\alpha }^{\beta }\mathcal{V}_{\beta}
\end{equation}
in order to obtain a local frame $\{\mathcal{X}_{\alpha },\mathcal{V}_{\alpha }\}$ on $\mathcal{T}^{\prime }E$, called the \emph{adapted frame} with respect to the induced complex nonlinear connection on $\mathcal{T}^{\prime }E$. Then 
\begin{equation*}
\rho _{\mathcal{T}}(\mathcal{Z}_{\alpha })=\partial_\alpha,\qquad \rho _{\mathcal{T}}(\mathcal{V}_{\alpha })=\dot{\partial}_\alpha, 
\end{equation*}
hence
\begin{equation*}
\rho _{\mathcal{T}}(\mathcal{X}_{\alpha })=\delta_\alpha.
\end{equation*}

In \cite{9}, locally imposing that the adapted frame must change by the rules 
\begin{equation}
\delta _{\alpha }=M_{\alpha }^{\beta }\widetilde{\delta }_{\beta }
\label{7.4}
\end{equation}
for local changes $\widetilde{z}^{k}=\widetilde{z}^{k}(z),\ \widetilde{u}^{\alpha }=M_{\beta }^{\alpha }(z)u^{\beta }$ on $E$, we have obtained that the coefficients of the nonlinear connection change by the rule:
\begin{equation}
M_{\alpha }^{\beta }\widetilde{N}_{\beta }^{\gamma }=M_{\beta }^{\gamma}N_{\alpha }^{\beta }-\rho _{\alpha }^{k}\dfrac{\partial M_{\beta }^{\gamma }}{\partial z^{k}}u^{\beta }.  \label{7.5}
\end{equation}

From \eqref{sch.c.e}, \eqref{sch.c.T'E} and \eqref{7.4} we get
\begin{align*}
\widetilde{\mathcal{X}}_\alpha &= \big( \widetilde{e}_\alpha,\widetilde{\delta}_\alpha \big) = \big( W^\beta_\alpha e_\beta, W^\beta_\alpha \delta_\beta \big),\\
\widetilde{\mathcal{V}}_\alpha &= \big( 0,\widetilde{\dot{\partial}}_\alpha \big) = \big( 0, W^\beta_\alpha \dot{\partial}_\beta \big),
\end{align*}
such that the rules of change for the adapted frame $\{\mathcal{X}_{\alpha },\mathcal{V}_{\alpha }\}$ are
\begin{align}
\label{sch.rep.prel}
\widetilde{\mathcal{X}}_\alpha &= W^\beta_\alpha\mathcal{X}_\beta,\\
\notag
\widetilde{\mathcal{V}}_\alpha &= W^\beta_\alpha\mathcal{V}_\beta.
\end{align}

A significant result in \cite{9} is Theorem 2.2, in which a complex nonlinear connection is determined from a semispray and hence, together with Theorem 2.1, a special complex nonlinear connection on the prolongation $\mathcal{T}^{\prime }E$ can be obtained if on $E$ is given a Lagrange (Finsler) function.

Obviously, on the complexified prolongation bundle, a complex nonlinear connection determines the splitting of $\mathcal{T}_{\mathbb{C}}E$ as 
\begin{equation}
\label{split}
\mathcal{T}_{\mathbb{C}}E = H\mathcal{T}_{\mathbb{C}}E \oplus V\mathcal{T}_{\mathbb{C}}E \oplus \overline{H\mathcal{T}_{\mathbb{C}}E} \oplus \overline{V\mathcal{T}_{\mathbb{C}}E}  
\end{equation}
such that an adapted frame $\{\mathcal{X}_{\alpha },\mathcal{V}_{\alpha },\mathcal{X}_{\bar{\alpha}},\mathcal{V}_{\bar{\alpha}}\}$ is obtained on $\mathcal{T}_{\mathbb{C}}E$ with respect to the complex nonlinear connection.

\begin{prop}
\label{brackets}
The Lie brackets of the adapted frame $\{\mathcal{X}_\alpha,\mathcal{V}_\alpha,\mathcal{X}_{\bar{\alpha}},\mathcal{V}_{\bar{\alpha}}\}$ are 
\begin{align*}
[\mathcal{X}_\alpha,\mathcal{X}_\beta]_{\mathcal{T}} &=
\mathcal{C}^{\:\gamma}_{\alpha\beta}\mathcal{X}_\gamma + \mathcal{R}^{\:\gamma}_{\alpha\beta}\mathcal{V}_\gamma, \\
[\mathcal{X}_\alpha,\mathcal{X}_{\bar{\beta}}]_{\mathcal{T}} &=  (\delta_{\bar{\beta}}N^\gamma_\alpha)\mathcal{V}_\gamma - (\delta_\alpha N^{\bar{\gamma}}_{\bar{\beta}})\mathcal{V}_{\bar{\gamma}}, \\
[\mathcal{X}_\alpha,\mathcal{V}_\beta]_{\mathcal{T}} &= (\dot{\partial}_\beta N^\gamma_\alpha)\mathcal{V}_\gamma, \\
[\mathcal{X}_\alpha,\mathcal{V}_{\bar{\beta}}]_{\mathcal{T}} &= (\dot{\partial}_{\bar{\beta}}N^\gamma_\alpha)\mathcal{V}_\gamma, \\
[\mathcal{V}_\alpha,\mathcal{V}_\beta]_{\mathcal{T}} &= 0, \\
[\mathcal{V}_\alpha,\mathcal{V}_{\bar{\beta}}]_{\mathcal{T}} &= 0,
\end{align*}
where 
\begin{equation*}
\mathcal{R}^{\:\gamma}_{\alpha\beta} = \mathcal{C}^{\:\varepsilon}_{\alpha\beta}N_\varepsilon^\gamma - \delta_\alpha N^\gamma_\beta + \delta_\beta N^\gamma_\alpha.
\end{equation*}
\end{prop}

\begin{proof}
We will only prove the first identity, the computations for the others are similar. We have:
\begin{align*}
[\mathcal{X}_\alpha,\mathcal{X}_\beta] &= [\mathcal{Z}_\alpha-N^\gamma_\alpha\mathcal{V}_\gamma, \mathcal{Z}_\beta-N^\delta_\beta\mathcal{V}_\delta]\\
&= [\mathcal{Z}_\alpha, \mathcal{Z}_\beta] - [\mathcal{Z}_\alpha, N^\delta_\beta\mathcal{V}_\delta] - [N^\gamma_\alpha\mathcal{V}_\gamma, \mathcal{Z}_\beta] + [N^\gamma_\alpha\mathcal{V}_\gamma, N^\delta_\beta\mathcal{V}_\beta]\\
&= \mathcal{C}^{\:\gamma}_{\alpha\beta}\mathcal{Z}_\gamma - N^\delta_\beta[\mathcal{Z}_\alpha, \mathcal{V}_\delta] - \rho_{\mathcal{T}}(\mathcal{Z}_\alpha)(N^\delta_\beta)\mathcal{V}_\delta - N^\gamma_\alpha[\mathcal{V}_\gamma, \mathcal{Z}_\beta]\\
&\quad +\rho_{\mathcal{T}}(\mathcal{Z}_\beta)(N^\gamma_\alpha)\mathcal{V}_\gamma + N^\gamma_\alpha N^\delta_\beta[\mathcal{V}_\gamma, \mathcal{V}_\delta] + N^\gamma_\alpha\rho_{\mathcal{T}}(\mathcal{V}_\gamma)(N^\delta_\beta)\mathcal{V}_\delta\\
&\quad - N^\delta_\beta\rho_{\mathcal{T}}(\mathcal{V}_\delta)(N^\gamma_\alpha)\mathcal{V}_\gamma\\
&= \mathcal{C}^{\:\gamma}_{\alpha\beta}\mathcal{Z}_\gamma - \rho^k_\alpha\dfrac{\partial N^\delta_\beta}{\partial z^k}\mathcal{V}_\delta + \rho^h_\beta\dfrac{\partial N^\gamma_\alpha}{\partial z^h}\mathcal{V}_\gamma + N^\gamma_\alpha\dfrac{\partial N^\delta_\beta}{\partial u^\gamma}\mathcal{V}_\delta - N^\delta_\beta\dfrac{\partial N^\gamma_\alpha}{\partial u^\delta}\mathcal{V}_\gamma\\
&= \mathcal{C}^{\:\gamma}_{\alpha\beta}\mathcal{Z}_\gamma - \rho^k_\alpha\dfrac{\partial N^\gamma_\beta}{\partial z^k}\mathcal{V}_\gamma + \rho^h_\beta\dfrac{\partial N^\gamma_\alpha}{\partial z^h}\mathcal{V}_\gamma + N^\varepsilon_\alpha\dfrac{\partial N^\gamma_\beta}{\partial u^\varepsilon}\mathcal{V}_\gamma - N^\varepsilon_\beta\dfrac{\partial N^\gamma_\alpha}{\partial u^\varepsilon}\mathcal{V}_\gamma\\
&= \mathcal{C}^{\:\gamma}_{\alpha\beta}\mathcal{X}_\gamma + \mathcal{C}^{\:\varepsilon}_{\alpha\beta}N^\gamma_\varepsilon\mathcal{V}_\gamma - \rho^k_\alpha\dfrac{\partial N^\gamma_\beta}{\partial z^k}\mathcal{V}_\gamma + \rho^h_\beta\dfrac{\partial N^\gamma_\alpha}{\partial z^h}\mathcal{V}_\gamma\\
&\quad + N^\varepsilon_\alpha\dfrac{\partial N^\gamma_\beta}{\partial u^\varepsilon}\mathcal{V}_\gamma - N^\varepsilon_\beta\dfrac{\partial N^\gamma_\alpha}{\partial u^\varepsilon}\mathcal{V}_\gamma\\
&= \mathcal{C}^{\:\gamma}_{\alpha\beta}\mathcal{X}_\gamma + \mathcal{C}^{\:\varepsilon}_{\alpha\beta}N^\gamma_\varepsilon\mathcal{V}_\gamma - \delta_\alpha N^\gamma_\beta\mathcal{V}_\gamma + \delta_\beta N^\gamma_\alpha\mathcal{V}_\gamma.
\end{align*}
\end{proof}

\section{Complex Finsler structures on a Lie algebroid}

The Finsler structure will be considered as a function defined on the complexified tangent bundle of a Lie algebroid $E$. The motivation for this choice is the desire to obtain properties that are similar to those from the case of the holomorphic tangent bundle $T'M$ (\cite{1,15}). First, denote by $\widetilde{E}$ the open submanifold of $E$ consisting in the nonzero sections.

\begin{defn}
A complex Finsler structure $F$ on $E$ is a real-valued function $F:E\rightarrow\mathbb{R}$ satisfying the following conditions:

\begin{description}
\item[1)] $F$ is $C^{\infty }$-class on $\widetilde{E}$;

\item[2)] $F(z,u) \geq 0$ and $F(z,u)=0$ iff $u=0$;

\item[3)] $F(z,\lambda u)=|\lambda |^{2}F(z,u)$ for all $\lambda \in \mathbb{C}$.
\end{description}
\end{defn}

Of course, since $F$ on $E$ is a real-valued function, it acts with the same rules on the whole complexified bundle $E_{\mathbb{C}}$.

As in the case of complex vector bundles (see \cite{4}), we will say that a Finsler structure $F$ is \emph{convex} if the Hermitian matrix defined in our case by
\begin{equation}
\label{h}
h_{\alpha \bar{\beta}}=\dot{\partial}_\alpha \dot{\partial}_{\bar{\beta}}F
\end{equation}
is positive-definite. In
the following, we will assume the convexity of $F$.

\begin{defn}
The pair $(E,F)$ is called the complex Finsler Lie algebroid.
\end{defn}

The tensor $h_{\alpha \bar{\beta}}$ defines a Hermitian metric $G$ on the vertical subbundle $VT_{\mathbb{C}}E$ by $G(Z,W)=h_{\alpha \bar{\beta}}Z^{\alpha }W^{\bar{\beta}}$, where $h_{\alpha \bar{\beta}}=G(\dot{\partial}_{\alpha },\dot{\partial}_{\bar{\beta}})$.

The following result contains important properties of the Finsler function $F$ that are consequences of the third property from its definition, i.e., the homogeneity of $F$.

\begin{prop}
The Finsler function $F$ on the algebroid $E$ satisfies:
\begin{description}
\item[i)] $(\dot{\partial}_\alpha F) u^\alpha = F\ $, $(\dot{\partial}_{\bar{\alpha}}F) \bar{u}^\alpha = F$;
\item[ii)] $h_{\alpha\bar{\beta}} u^\alpha = \dot{\partial}_{\bar{\beta}}F\ $, $h_{\alpha\bar{\beta}} u^{\bar{\beta}} = \dot{\partial}_\alpha F\ $, $F = h_{\alpha\bar{\beta}}u^\alpha\bar{u}^\beta$;
\item[iii)] $(\dot{\partial}_\gamma h_{\alpha\bar{\beta}})u^\gamma = 0\ $, $(\dot{\partial}_\gamma h_{\alpha\bar{\beta}})u^\alpha = 0\ $, $(\dot{\partial}_{\bar{\gamma}} h_{\alpha\bar{\beta}})\bar{u}^\gamma = 0$;
\item[iv)] $h_{\alpha\beta}u^\alpha = 0\ $, $(\dot{\partial}_\gamma h_{\alpha\bar{\beta}})\bar{u}^\beta = h_{\alpha\gamma}\ $, where $h_{\alpha\beta} = \dot{\partial}_\alpha\dot{\partial}_\beta F$.
\end{description}
\end{prop}

\subsection{The Chern-Finsler connection}

The most well-known linear connection in Finsler geometry is the Chern-Finsler connection. In this section, we introduce such a connection on the Lie algebroid $E$.

On a complex vector bundle, the notion of normal complex linear connection does not make sense, due to the fact that the rules of change of the coefficients of a distinguished linear connection do not coincide in pairs, such as in the case of $T'M$. This is well-known from \cite{15}. But the classical Chern-Finsler connection is a normal complex linear connection. Therefore, we shall induce a Chern-Finsler linear connection on the prolongation $\mathcal{T}_{\mathbb{C}}E$ starting from a vertical connection on $E$.

As in the case of complex Finsler vector bundles (\cite{2,3,5,15}), we shall consider the functions $N_k^\beta(z,u)$ defined by 
\begin{equation}
\label{cnct.C-F}
N_k^\beta = h^{\bar{\sigma}\beta }\partial_k\dot{\partial}_{\bar{\sigma}}F.
\end{equation}
Then, using \eqref{rel.coef.N} and \eqref{act.ro.prel}, we get
\begin{equation}
\label{cncp.C-F}
N_\alpha^\beta = h^{\bar{\sigma}\beta }\rho_k^\alpha\partial_k\dot{\partial}_{\bar{\sigma}}F = h^{\bar{\sigma}\beta }\partial_\alpha\dot{\partial}_{\bar{\sigma}}F,
\end{equation}
the coefficients of a nonlinear connection on $\mathcal{T}_{\mathbb{C}}E$.

\begin{prop}
The functions $N_\alpha^\beta(z,u)$ defined by (\ref{cncp.C-F}) determine a nonlinear connection on $\mathcal{T}_{\mathbb{C}}E$, called the Chern-Finsler nonlinear connection of the Lie algebroid $E$.
\end{prop}
\begin{proof}
It suffices to check that $N_\alpha^\beta$ satisfy the \eqref{7.5} rule of change. First, we have that
\begin{equation}
\widetilde{N}^\gamma_\beta = \widetilde{h}^{\bar{\sigma}\gamma}\widetilde{\rho}^j_\beta\widetilde{\partial}_j\widetilde{\dot{\partial}}_{\bar{\sigma}}F.
\end{equation}
Then, from \eqref{h} we get
\begin{equation*}
\widetilde{h}^{\bar{\varepsilon}\gamma} = M^\gamma_\mu M^{\bar{\varepsilon}}_{\bar{\nu}}h^{\bar{\nu}\mu},
\end{equation*}
which, replaced in the left-hand side of \eqref{7.5} together with \eqref{sch.c.ro} and \eqref{sch.c.T'E}, yields after some basic computations the right-hand side.
\end{proof}

Now, consider a partial linear connection on the vertical subbundle of the tangent bundle of $E$, $\mathcal{D}:T_{\mathbb{C}}E\times VT_{\mathbb{C}}E\rightarrow VT_{\mathbb{C}}E$, which preserves the distributions $VT_{\mathbb{C}}E$ and $\overline{VT_{\mathbb{C}}E}$ and commutes with the conjugation. Its local coefficients are 
\begin{align}
& D_{\delta_k}\dot{\partial}_{\alpha } = L_{\alpha k}^{\:\gamma }\dot{\partial}_{\gamma },\quad D_{\dot{\partial}_{\beta }}
\dot{\partial}_{\alpha } = C_{\alpha\beta }^{\:\gamma}\dot{\partial}_\gamma ,
\label{coef.D} \\
& D_{\delta_k}\dot{\partial}_{\bar{\alpha}} = L_{\bar{\alpha}k}^{\:\bar{\gamma}}\dot{\partial}_{\bar{\gamma}},\quad D_{\dot{\partial}_{\beta }}\dot{\partial}_{\bar{\alpha}} = C_{\bar{\alpha}\beta }^{\:\bar{\gamma}}\dot{\partial}_{\bar{\gamma}}  \notag
\end{align}
and their conjugates. The rules of change for their coefficients are:
\begin{align}
\notag
\widetilde{L}^{\:\tau}_{\alpha k} &= M^\tau_\gamma \dfrac{\partial z^h}{\partial\widetilde{z}^k} \bigg[ \dfrac{\partial W^\gamma_\alpha}{\partial z^h} + W^\theta_\alpha L^{\:\gamma}_{\theta h} \bigg], \\
\label{sch.L,C1}
\widetilde{C}^{\:\tau}_{\alpha\beta} &= M^\tau_\gamma W^\sigma_\beta W^\theta_\alpha C^{\:\gamma}_{\theta\sigma}, \\
\notag
\widetilde{L}^{\:\bar{\tau}}_{\bar{\alpha}k} &= \dfrac{\partial z^h}{\partial\widetilde{z}^k} M^{\bar{\tau}}_{\bar{\gamma}} W^{\bar{\theta}}_{\bar{\alpha}} L^{\:\bar{\gamma}}_{\bar{\theta}h}, \\
\notag
\widetilde{C}^{\:\bar{\tau}}_{\bar{\alpha}\beta} &= M^{\bar{\tau}}_{\bar{\gamma}} W^\sigma_\beta W^{\bar{\theta}}_{\bar{\alpha}} C^{\:\bar{\gamma}}_{\bar{\theta}\sigma}.
\end{align}
The proof for these identities is very similar to the one for the coefficients of the linear connection which will be defined in the following, therefore we will not give these computations here.

As in the case of a complex vector bundle (\cite{3,15}), we will consider
\begin{equation}
\label{L,C}
L_{\alpha k}^{\:\gamma } = h^{\bar{\sigma}\gamma}\delta_kh_{\alpha\bar{\sigma}},\quad C_{\alpha\beta }^{\:\gamma} = h^{\bar{\sigma}\gamma}\dot{\partial}_\beta h_{\alpha\bar{\sigma}}
\end{equation}
together with their conjugates. It is easy to check that they satisfy the \eqref{sch.L,C1} rules of change.

\begin{prop}
The following identity holds:
\begin{equation}
\label{rel.L,N}
L^{\:\gamma}_{\alpha k} = \dot{\partial}_\alpha N^\gamma_k.
\end{equation}
\end{prop}
\begin{proof}
From $h^{\bar{\sigma}\tau}h_{\tau\bar{\varepsilon}} = \delta^{\bar{\sigma}}_{\bar{\varepsilon}}$, we get
\begin{equation*}
\dfrac{\partial h^{\bar{\sigma}\gamma}}{\partial u^\alpha} = -h^{\bar{\varepsilon}\gamma}h^{\bar{\sigma}\tau}\dfrac{\partial h_{\tau\bar{\varepsilon}}}{\partial u^\alpha}.
\end{equation*}
Using also the properties derived from the homogeneity of $F$, we have:
\begin{align*}
\dot{\partial}_\alpha N^\gamma_k
&= \dfrac{\partial}{\partial u^\alpha}\bigg( h^{\bar{\sigma}\gamma}\dfrac{\partial^2F}{\partial z^k\partial\bar{u}^\sigma} \bigg)\\
&= \dfrac{\partial h^{\bar{\sigma}\gamma}}{\partial u^\alpha}\dfrac{\partial^2F}{\partial z^k \partial \bar{u}^\sigma} + h^{\bar{\sigma}\gamma}\dfrac{\partial h_{\alpha\bar{\sigma}}}{\partial z^k}\\
&= h^{\bar{\sigma}\gamma}\dfrac{\partial h_{\alpha\bar{\sigma}}}{\partial z^k} - h^{\bar{\varepsilon}\gamma}h^{\bar{\sigma}\tau}\dfrac{\partial h_{\tau\bar{\varepsilon}}}{\partial u^\alpha}\dfrac{\partial^2F}{\partial z^k \partial \bar{u}^\sigma}\\
&= h^{\bar{\sigma}\gamma}\dfrac{\partial h_{\alpha\bar{\sigma}}}{\partial z^k} - h^{\bar{\varepsilon}\gamma}\dfrac{\partial h_{\tau\bar{\varepsilon}}}{\partial u^\alpha} N^\tau_k\\
&= h^{\bar{\sigma}\gamma}\dfrac{\partial h_{\alpha\bar{\sigma}}}{\partial z^k} - h^{\bar{\sigma}\gamma}N^\alpha_k\dfrac{\partial h_{\alpha\bar{\sigma}}}{\partial u^\tau} \\
&= h^{\bar{\sigma}\gamma}\dfrac{\delta h_{\alpha\bar{\sigma}}}{\delta z^k} = L^{\:\gamma}_{\alpha k}.
\end{align*}
\end{proof}

We now introduce an $N$-complex linear connection on the complexified prolongation of $E$, $\mathcal{D}:\mathcal{T}_{\mathbb{C}}E\times \mathcal{T}_{\mathbb{C}}E\rightarrow \mathcal{T}_{\mathbb{C}}E$, by 
\begin{align}
& \mathcal{D}_{\mathcal{X}_\beta }\mathcal{V}_{\alpha } = L_{\alpha \beta }^{\:\gamma }\mathcal{V}_{\gamma },\quad \mathcal{D}_{\mathcal{V}_{\beta }}\mathcal{V}_{\alpha } = C_{\alpha\beta }^{\:\gamma}\mathcal{V}_{\gamma},
\label{coef.D'} \\
& \mathcal{D}_{\mathcal{X}_\beta}\mathcal{V}_{\bar{\alpha}} = L_{\bar{\alpha}\beta }^{\:\bar{\gamma}}\mathcal{V}_{\bar{\gamma}},\quad \mathcal{D}_{\mathcal{V}_{\beta }}\mathcal{V}_{\bar{\alpha}} = C_{\bar{\alpha}\beta }^{\:\bar{\gamma}}\mathcal{V}_{\bar{\gamma}}  \notag
\end{align}
and, since $\mathcal{D}$ is a normal connection, it must preserve the distributions, that is,
\begin{align*}
& \mathcal{D}_{\mathcal{X}_\beta }\mathcal{X}_{\alpha } = L_{\alpha \beta }^{\:\gamma }\mathcal{X}_{\gamma },\quad \mathcal{D}_{\mathcal{V}_{\beta }}\mathcal{X}_{\alpha } = C_{\alpha\beta }^{\:\gamma}\mathcal{X}_{\gamma},\\
& \mathcal{D}_{\mathcal{X}_\beta}\mathcal{X}_{\bar{\alpha}} = L_{\bar{\alpha}\beta }^{\:\bar{\gamma}}\mathcal{X}_{\bar{\gamma}},\quad \mathcal{D}_{\mathcal{V}_{\beta }}\mathcal{X}_{\bar{\alpha}} = C_{\bar{\alpha}\beta }^{\:\bar{\gamma}}\mathcal{X}_{\bar{\gamma}}.
\end{align*}

\begin{prop}
The rules of change for the coefficients of the connection $\mathcal{D}$ are:
\begin{align}
\notag
\widetilde{L}^{\:\tau}_{\alpha\beta} &= M^\tau_\gamma W^\sigma_\beta \big[ \rho^k_\sigma(\partial_kW^\gamma_\alpha) + W^\theta_\alpha L^{\:\gamma}_{\theta\sigma} \big],\\
\label{sch.L,C2}
\widetilde{C}^{\:\tau}_{\alpha\beta} &= M^\tau_\gamma W^\sigma_\beta W^\theta_\alpha C^{\:\gamma}_{\theta\sigma} \\
\notag
\widetilde{L}^{\:\bar{\tau}}_{\bar{\alpha}\beta} &= M^{\bar{\tau}}_{\bar{\gamma}} W^\sigma_\beta W^{\bar{\theta}}_{\bar{\alpha}} L^{\:\bar{\gamma}}_{\bar{\theta}\sigma},\\
\notag
\widetilde{C}^{\:\bar{\tau}}_{\bar{\alpha}\beta} &= M^{\bar{\tau}}_{\bar{\gamma}} W^\sigma_\beta W^{\bar{\theta}}_{\bar{\alpha}} C^{\:\bar{\gamma}}_{\bar{\theta}\sigma}.
\end{align}
\end{prop}
\begin{proof}
We only prove the first identity. Using \eqref{sch.rep.prel}, we have:
\begin{align*}
\mathcal{D}_{\widetilde{\mathcal{X}}_\beta}\widetilde{\mathcal{X}}_\alpha
&= \mathcal{D}_{W^\sigma_\beta\mathcal{X}_\sigma}(W^\gamma_\alpha\mathcal{X}_\gamma)\\
&= W^\sigma_\beta \big[ \big( \rho_{\mathcal{T}}(\mathcal{X}_\sigma)W^\gamma_\alpha \big)\mathcal{X}_\gamma + W^\gamma_\alpha \mathcal{D}_{\mathcal{X}_\sigma}\mathcal{X}_\gamma \big]\\
&= W^\sigma_\beta \big[ \rho^k_\sigma(\delta_k W^\gamma_\alpha)\mathcal{X}_\gamma + W^\gamma_\alpha L^{\:\theta}_{\gamma\sigma} \mathcal{X}_\theta \big]\\
&= W^\sigma_\beta \big[ \rho^k_\sigma(\partial_k W^\gamma_\alpha) + W^\theta_\alpha L^{\:\gamma}_{\theta\sigma} \big] \mathcal{X}_\gamma.
\end{align*}
On the other hand,
\begin{equation*}
\mathcal{D}_{\widetilde{\mathcal{X}}_\beta}\widetilde{\mathcal{X}}_\alpha = \widetilde{L}^{\:\sigma}_{\alpha\beta}W^\gamma_\sigma\mathcal{X}_\gamma,
\end{equation*}
and identifying the coefficients gives the first rule of change. The others can be proved in a similar manner.
\end{proof}

The relation between the coefficients of the two linear connections $D$ on $T_{\mathbb{C}}E$ and $\mathcal{D}$ on $\mathcal{T}_{\mathbb{C}}E$ is given by the following
\begin{lem}
The functions $L^{\:\gamma}_{\alpha\beta}$ given by
\begin{equation}
\label{rel.L}
L^{\:\gamma}_{\alpha\beta} = \rho^k_\beta L^{\:\gamma}_{\alpha k}
\end{equation}
are the coefficients of a normal complex linear connection on $\mathcal{T}_{\mathbb{C}}E$.
\end{lem}
\begin{proof}
It consists of checking that the functions $L^{\:\gamma}_{\alpha\beta}$ satisfy the first rule from Proposition 2.4, which is immediate.
\end{proof}

Let $\mathcal{T}(Z,W) = \mathcal{D}_ZW - \mathcal{D}_WZ - [Z,W]_{\mathcal{T}}$ be the torsion of the connection $\mathcal{D}$ on $\mathcal{T}_{\mathbb{C}}E$ and denote by $h\mathcal{T}(Z,W)$, $v\mathcal{T}(Z,W)$, $\bar{h}\mathcal{T}(Z,W)$ and $\bar{v}\mathcal{T}(Z,W)$ its components with respect to the \eqref{split} splitting. A basic computation using Proposition \ref{brackets} leads to: 
\begin{align*}
\mathcal{T}(\mathcal{X}_\alpha,\mathcal{X}_\beta) &= \big(L^{\:\gamma}_{\beta\alpha}-L^{\:\gamma}_{\alpha\beta}-\mathcal{C}^{\:\gamma}_{\alpha\beta}\big)\mathcal{X}_\gamma - \mathcal{R}^{\:\gamma}_{\alpha\beta}\mathcal{V}_\gamma,\\
\mathcal{T}(\mathcal{X}_\alpha,\mathcal{X}_{\bar{\beta}}) &= -L^{\:\gamma}_{\alpha\bar{\beta}}\mathcal{X}_\gamma + L^{\:\bar{\gamma}}_{\bar{\beta}\alpha}\mathcal{X}_{\bar{\gamma}}  -(\delta_{\bar{\beta}}N^\gamma_\alpha)\mathcal{V}_\gamma + (\delta_\alpha N^{\bar{\gamma}}_{\bar{\beta}})\mathcal{V}_{\bar{\gamma}},\\
\mathcal{T}(\mathcal{X}_\alpha,\mathcal{V}_\beta) &= -C^{\:\gamma}_{\alpha\beta}\mathcal{X}_\gamma + \big(L^{\:\gamma}_{\beta\alpha}-(\dot{\partial}_\beta N^\gamma_\alpha)\big)\mathcal{V}_\gamma,\\
\mathcal{T}(\mathcal{X}_\alpha,\mathcal{V}_{\bar{\beta}}) &= -C^{\:\gamma}_{\alpha\bar{\beta}}\mathcal{X}_\gamma - (\dot{\partial}_{\bar{\beta}}N^\gamma_\alpha)\mathcal{V}_\gamma + L^{\:\bar{\gamma}}_{\bar{\beta}\alpha}\mathcal{V}_{\bar{\gamma}},\\
\mathcal{T}(\mathcal{V}_\alpha,\mathcal{V}_\beta) &= \big(C^{\:\gamma}_{\beta\alpha} - C^{\:\gamma}_{\alpha\beta}\big)\mathcal{V}_\gamma,\\
\mathcal{T}(\mathcal{V}_\alpha,\mathcal{V}_{\bar{\beta}}) &= -C^{\:\bar{\gamma}}_{\alpha\bar{\beta}}\mathcal{V}_\gamma + C^{\:\bar{\gamma}}_{\bar{\beta}\alpha}\mathcal{V}_{\bar{\gamma}}
\end{align*}
and also their conjugates. In a similar manner we can compute the local expressions of the curvature of the connection $\mathcal{D}$, defined by $\mathcal{R}(Z,W)V = \mathcal{D}_Z\mathcal{D}_WV - \mathcal{D}_W\mathcal{D}_ZV - \mathcal{D}_{[Z,W]}V$.

Let us now impose that $\mathcal{D}$ is a $(1,0)$-connection, i.e., $L_{\bar{\alpha}\beta }^{\:\bar{\gamma}}=C_{\bar{\alpha}\beta }^{\:\bar{\gamma}}=0$, and consider the functions $L^{\:\gamma}_{\alpha k}$ from \eqref{L,C}. Using \eqref{rel.L}, we get the following
\begin{thm}
The functions $L^{\:\gamma}_{\alpha\beta}$ given by
\begin{equation}
\label{Lprel}
L^{\:\gamma}_{\alpha\beta} = \rho^k_\beta L^{\:\gamma}_{\alpha k} = h^{\bar{\sigma}\gamma}\delta_\beta h_{\alpha\bar{\sigma}}
\end{equation}
are the coefficients of a linear connection of $(1,0)$-type on $\mathcal{T}_{\mathbb{C}}E$, called the Chern-Finsler connection of the algebroid $E$. It is induced by the vertical connection \eqref{L,C}.
\end{thm}

\begin{rem}
The theorem can also be proved directly by checking that \eqref{Lprel} satisfy the first rule of change in \eqref{sch.L,C2}.
\end{rem}

We can conclude that the Chern-Finsler connection of the Lie algebroid $E$ is given by
\begin{equation}
\label{C-F}
N_\alpha^\beta = h^{\bar{\sigma}\beta }\partial_\alpha\dot{\partial}_{\bar{\sigma}}F,\quad L^{\:\gamma}_{\alpha\beta} = h^{\bar{\sigma}\gamma}\delta_\beta h_{\alpha\bar{\sigma}},\quad C_{\alpha\beta }^{\:\gamma} = h^{\bar{\sigma}\gamma}\dot{\partial}_\beta h_{\alpha\bar{\sigma}}
\end{equation}
and $L_{\bar{\alpha}\beta }^{\:\bar{\gamma}}=C_{\bar{\alpha}\beta }^{\:\bar{\gamma}}=0$. Also, we note that
\begin{equation}
\label{C}
C^{\:\gamma}_{\alpha\beta} = C^{\:\gamma}_{\beta\alpha}
\end{equation}
and, due to \eqref{rel.L} and \eqref{rel.L,N},
\begin{equation}
\label{L,N}
L^{\:\gamma}_{\alpha\beta} = \dot{\partial}_\alpha N^\gamma_\beta.
\end{equation}

The Chern-Finsler connection of $E$ has a similar property to the classical Chern-Finsler connection.

\begin{prop}
The adapted frame of fields corresponding to the Chern-Finsler connection \eqref{cncp.C-F} satisfies the identity 
\begin{equation} 
\label{cr.C-F}
\rho_{\mathcal{T}}\big([\mathcal{X}_\alpha,\mathcal{X}_\beta]\big) = 0.
\end{equation}
\end{prop}
\begin{proof}
We have $\rho_{\mathcal{T}}\big([\mathcal{X}_\alpha,\mathcal{X}_\beta]\big) = [\rho_{\mathcal{T}}(\mathcal{X}_\alpha),\rho_{\mathcal{T}}(\mathcal{X}_\beta)] = [\delta_\alpha,\delta_\beta]$. Then, using $h^{\bar{\varepsilon}\gamma}h_{\tau\bar{\varepsilon}} = \delta^\gamma_\tau$, we get $\delta_\alpha(h^{\bar{\mu}\gamma}) = -h^{\bar{\mu}\tau}h^{\bar{\varepsilon}\gamma}\delta_\alpha(h_{\tau\bar{\varepsilon}})$. We can now compute the coefficient of index $\sigma$ of the Lie bracket $[\delta_\alpha,\delta_\beta]$:
\begin{align*}
[\delta_\alpha,\delta_\beta]^\sigma 
&= \delta_\beta(N^\sigma_\alpha) - \delta_\alpha(N^\sigma_\beta)\\
&= \partial_\beta(N^\sigma_\alpha) - \partial_\alpha(N^\sigma_\beta) + N^\delta_\alpha\dot{\partial}_\delta(N^\sigma_\beta) - N^\nu_\beta\dot{\partial}_\nu(N^\sigma_\alpha)\\
&= \partial_\beta(h^{\bar{\delta}\sigma})\partial_\alpha\dot{\partial}_{\bar{\delta}}F - \partial_\alpha(h^{\bar{\mu}\sigma})\partial_\beta\dot{\partial}_{\bar{\mu}}F\\
&\quad + h^{\bar{\theta}\delta}\partial_\alpha\dot{\partial}_{\bar{\theta}}F\left[\dot{\partial}_\delta(h^{\bar{\mu}\sigma})\partial_\beta\dot{\partial}_{\bar{\mu}}F + h^{\bar{\mu}\sigma}\partial_\beta(h_{\delta\bar{\mu}})\right]\\
&\quad -h^{\bar{\theta}\nu}\partial_\beta\dot{\partial}_{\bar{\theta}}F\left[\dot{\partial}_\nu(h^{\bar{\delta}\sigma})\partial_\alpha\dot{\partial}_{\bar{\delta}}F + h^{\bar{\delta}\sigma}\partial_\alpha(h_{\nu\bar{\delta}})\right]\\
&= h^{\bar{\theta}\delta}\dot{\partial}_\delta(h^{\bar{\mu}\sigma})\partial_\alpha\dot{\partial}_{\bar{\theta}}F\;\partial_\beta\dot{\partial}_{\bar{\mu}}F - h^{\bar{\theta}\nu}\dot{\partial}_\nu(h^{\bar{\gamma}\sigma})\partial_\beta\dot{\partial}_{\bar{\theta}}F\;\partial_\alpha\dot{\partial}_{\bar{\gamma}}F\\
&= N^\delta_\alpha N^\gamma_\beta h_{\gamma\bar{\mu}}\dot{\partial}_\delta(h^{\bar{\mu}\sigma}) - N^\nu_\beta N^\gamma_\alpha h_{\gamma\bar{\delta}}\dot{\partial}_\nu(h^{\bar{\delta}\sigma})\\
&= N^\delta_\alpha N^\gamma_\beta \left[ h_{\gamma\bar{\mu}}\dot{\partial}_\delta(h^{\bar{\mu}\sigma}) - h_{\delta\bar{\mu}}\dot{\partial}_\gamma(h^{\bar{\mu}\sigma}) \right]\\
&= N^\delta_\alpha N^\gamma_\beta \left[ h_{\delta\bar{\mu}}h^{\bar{\mu}\tau}h^{\bar{\varepsilon}\sigma}\dot{\partial}_\gamma(h_{\tau\bar{\varepsilon}}) - h_{\gamma\bar{\mu}}h^{\bar{\mu}\tau}h^{\bar{\varepsilon}\sigma}\dot{\partial}_\delta(h_{\tau\bar{\varepsilon}}) \right]\\
&= N^\delta_\alpha N^\gamma_\beta h^{\bar{\varepsilon}\sigma} \left[ \dot{\partial}_\gamma(h_{\delta\bar{\varepsilon}}) - \dot{\partial}_\delta(h^{\gamma\bar{\varepsilon}}) \right]\\
&= 0.
\end{align*}
\end{proof}

Using \eqref{L,N} and \eqref{C}, the components of the torsion of the Chern-Finsler connection on $\mathcal{T}E$ become:
\begin{align*}
\mathcal{T}(\mathcal{X}_\alpha,\mathcal{X}_\beta) &= \big(L^{\:\gamma}_{\beta\alpha}-L^{\:\gamma}_{\alpha\beta}-\mathcal{C}^{\:\gamma}_{\alpha\beta}\big)\mathcal{X}_\gamma - \mathcal{R}^{\:\gamma}_{\alpha\beta}\mathcal{V}_\gamma,\\
\mathcal{T}(\mathcal{X}_\alpha,\mathcal{X}_{\bar{\beta}}) &= -(\delta_{\bar{\beta}}N^\gamma_\alpha)\mathcal{V}_\gamma + (\delta_\alpha N^{\bar{\gamma}}_{\bar{\beta}})\mathcal{V}_{\bar{\gamma}},\\
\mathcal{T}(\mathcal{X}_\alpha,\mathcal{V}_\beta) &= -C^{\:\gamma}_{\alpha\beta}\mathcal{X}_\gamma,\\
\mathcal{T}(\mathcal{X}_\alpha,\mathcal{V}_{\bar{\beta}}) &= -(\dot{\partial}_{\bar{\beta}}N^\gamma_\alpha)\mathcal{V}_\gamma,\\
\mathcal{T}(\mathcal{V}_\alpha,\mathcal{V}_\beta) &= 0,\\
\mathcal{T}(\mathcal{V}_\alpha,\mathcal{V}_{\bar{\beta}}) &= 0.
\end{align*}

The nonzero components of the curvature of the Chern-Finsler connection are:
\begin{align*}
\mathcal{R}(\mathcal{X}_\alpha,\mathcal{X}_\beta)\mathcal{X}_\gamma &= \big[\delta_\alpha L^{\:\tau}_{\gamma\beta} - \delta_\beta L^{\:\tau}_{\gamma\alpha} + L^{\:\sigma}_{\gamma\beta}L^{\:\tau}_{\sigma\alpha} - L^{\:\sigma}_{\gamma\alpha}L^{\:\tau}_{\sigma\beta} - \mathcal{C}^{\:\sigma}_{\alpha\beta}L^{\:\tau}_{\gamma\sigma} - \mathcal{R}^{\:\sigma}_{\alpha\beta}C^{\:\tau}_{\gamma\sigma}\big]\mathcal{X}_\tau,\\
\mathcal{R}(\mathcal{X}_\alpha,\mathcal{X}_{\bar{\beta}})\mathcal{X}_\gamma &= \big[-\delta_{\bar{\beta}}L^{\:\tau}_{\gamma\alpha} - \mathcal{R}^{\:\sigma}_{\alpha\bar{\beta}}C^{\:\tau}_{\gamma\sigma}\big]\mathcal{X}_\tau,\\
\mathcal{R}(\mathcal{X}_{\bar{\alpha}},\mathcal{X}_\beta)\mathcal{X}_\gamma &= \big[\delta_{\bar{\alpha}}L^{\:\tau}_{\gamma\beta} - \mathcal{R}^{\:\sigma}_{\bar{\alpha}\beta}C^{\:\tau}_{\gamma\sigma}\big]\mathcal{X}_\tau,\\
\mathcal{R}(\mathcal{X}_\alpha,\mathcal{X}_\beta)\mathcal{V}_\gamma &= \big[\delta_\alpha L^{\:\tau}_{\gamma\beta} - \delta_\beta L^{\:\tau}_{\gamma\alpha} + L^{\:\sigma}_{\gamma\beta}L^{\:\tau}_{\sigma\alpha} - L^{\:\sigma}_{\gamma\alpha}L^{\:\tau}_{\sigma\beta} - \mathcal{C}^{\:\sigma}_{\alpha\beta}L^{\:\tau}_{\gamma\sigma} - \mathcal{R}^{\:\sigma}_{\alpha\beta}C^{\:\tau}_{\gamma\sigma}\big]\mathcal{V}_\tau,\\
\mathcal{R}(\mathcal{X}_\alpha,\mathcal{X}_{\bar{\beta}})\mathcal{V}_\gamma &= \big[-\delta_{\bar{\beta}}L^{\:\tau}_{\gamma\alpha} - \mathcal{R}^{\:\sigma}_{\alpha\bar{\beta}}C^{\:\tau}_{\gamma\sigma}\big]\mathcal{V}_\tau,\\
\mathcal{R}(\mathcal{X}_{\bar{\alpha}},\mathcal{X}_\beta)\mathcal{V}_\gamma &= \big[\delta_{\bar{\alpha}}L^{\:\tau}_{\gamma\beta} - \mathcal{R}^{\:\sigma}_{\bar{\alpha}\beta}C^{\:\tau}_{\gamma\sigma}\big]\mathcal{V}_\tau,\\
\mathcal{R}(\mathcal{X}_\alpha,\mathcal{V}_\beta)\mathcal{V}_\gamma &= \big[\delta_\alpha C^{\:\tau}_{\gamma\beta} - \dot{\partial}_\beta L^{\:\tau}_{\gamma\alpha} + C^{\:\sigma}_{\gamma\beta}L^{\:\tau}_{\sigma\alpha} - L^{\:\sigma}_{\gamma\alpha}C^{\:\tau}_{\sigma\beta} - L^{\:\sigma}_{\beta\alpha}C^{\:\tau}_{\gamma\sigma}\big]\mathcal{V}_\tau,\\
\mathcal{R}(\mathcal{X}_\alpha,\mathcal{V}_{\bar{\beta}})\mathcal{V}_\gamma &= \big[-\dot{\partial}_{\bar{\beta}}L^{\:\tau}_{\gamma\alpha} - (\dot{\partial}_{\bar{\beta}}N^\sigma_\alpha)C^{\:\tau}_{\gamma\sigma}\big]\mathcal{V}_\tau,\\
\mathcal{R}(\mathcal{X}_{\bar{\alpha}},\mathcal{V}_\beta)\mathcal{V}_\gamma &= (\delta_{\bar{\alpha}}C^{\:\tau}_{\gamma\beta})\mathcal{V}_\tau,\\
\mathcal{R}(\mathcal{V}_\alpha,\mathcal{V}_\beta)\mathcal{X}_\gamma &= \big[\dot{\partial}_\alpha C^{\:\tau}_{\gamma\beta} - \dot{\partial}_\beta C^{\:\tau}_{\gamma\alpha} + C^{\:\sigma}_{\gamma\beta}C^{\:\tau}_{\sigma\alpha} - C^{\:\sigma}_{\gamma\alpha}C^{\:\tau}_{\sigma\beta}\big]\mathcal{X}_\tau,\\
\mathcal{R}(\mathcal{V}_\alpha,\mathcal{V}_{\bar{\beta}})\mathcal{X}_\gamma &= (-\dot{\partial}_{\bar{\beta}}C^{\:\tau}_{\gamma\alpha})\mathcal{X}_\tau,\\
\mathcal{R}(\mathcal{V}_{\bar{\alpha}},\mathcal{V}_\beta)\mathcal{X}_\gamma &= (\dot{\partial}_{\bar{\alpha}}C^{\:\tau}_{\gamma\beta})\mathcal{X}_\tau,\\
\mathcal{R}(\mathcal{V}_\alpha,\mathcal{V}_\beta)\mathcal{V}_\gamma &= \big[\dot{\partial}_\alpha C^{\:\tau}_{\gamma\beta} - \dot{\partial}_\beta C^{\:\tau}_{\gamma\alpha} + C^{\:\sigma}_{\gamma\beta}C^{\:\tau}_{\sigma\alpha} - C^{\:\sigma}_{\gamma\alpha}C^{\:\tau}_{\sigma\beta}\big]\mathcal{V}_\tau,\\
\mathcal{R}(\mathcal{V}_\alpha,\mathcal{V}_{\bar{\beta}})\mathcal{V}_\gamma &= (-\dot{\partial}_{\bar{\beta}}C^{\:\tau}_{\gamma\alpha}))\mathcal{V}_\tau,\\
\mathcal{R}(\mathcal{V}_{\bar{\alpha}},\mathcal{V}_\beta)\mathcal{V}_\gamma &= (\dot{\partial}_{\bar{\alpha}}C^{\:\tau}_{\gamma\beta})\mathcal{V}_\tau.
\end{align*}

We shall now analyze the dual of the adapted frame $\{\mathcal{X}_\alpha,\mathcal{V}_\alpha\}$ on $\mathcal{T}'E$. Denote the dual frame by $\{\mathcal{Z}^\alpha,\delta\mathcal{V}^\alpha\}$, where $\{\mathcal{Z}^\alpha,\mathcal{V}^\alpha\}$ is the dual frame of $\{\mathcal{Z}_\alpha,\mathcal{V}_\alpha\}$ and $\delta\mathcal{V}^\alpha = \mathcal{V}^\alpha+N^\alpha_\beta\mathcal{Z}^\beta$. Since the dual frame $\{\mathcal{Z}^\alpha,\mathcal{V}^\alpha\}$ changes by the rules
\begin{align*}
\widetilde{\mathcal{Z}}^\alpha &= M^\alpha_\beta \mathcal{Z}^\beta,\\
\widetilde{\mathcal{V}}^\alpha &= M^\alpha_\beta \mathcal{V}^\beta + \rho^k_\beta\dfrac{\partial M^\alpha_\gamma}{\partial z^k}u^\gamma\mathcal{Z}^\beta,
\end{align*}
we obtain
\begin{equation*}
\widetilde{\delta\mathcal{V}}^\alpha = M^\alpha_\beta \delta\mathcal{V}^\beta.
\end{equation*}

The differential of a function $f$ on the complexified prolongation $\mathcal{T}_{\mathbb{C}}E$ is locally expressible as
\begin{equation*}
df = (\delta_\alpha f)\mathcal{Z}^\alpha + (\dot{\partial}_\alpha f)\delta\mathcal{V}^\alpha + (\delta_{\bar{\alpha}}f)\mathcal{Z}^{\bar{\alpha}} + (\dot{\partial}_{\bar{\alpha}}f)\delta\mathcal{V}^{\bar{\alpha}}.
\end{equation*}
With respect to the \eqref{split} decomposition of the prolongation, the differential can be written as
\begin{equation*}
df = \partial^hf + \partial^vf + \bar{\partial}^hf + \bar{\partial}^vf,
\end{equation*}
where
\begin{align*}
\partial^hf &= (\delta_\alpha f)\mathcal{Z}^\alpha = \bigg(\rho^k_\alpha\dfrac{\partial f}{\partial z^k} - N^\beta_\alpha\dfrac{\partial f}{\partial u^\beta}\bigg)\mathcal{Z}^\alpha,\quad
\partial^vf &= (\dot{\partial}_\alpha f)\delta\mathcal{V}^\alpha = \dfrac{\partial f}{\partial u^\alpha}\delta\mathcal{V}^\alpha,\\
\bar{\partial}^hf &= (\delta_{\bar{\alpha}}f)\bar{\mathcal{Z}}^\alpha = \bigg(\rho^{\bar{k}}_{\bar{\alpha}}\dfrac{\partial f}{\partial\bar{z}^k} - N^{\bar{\beta}}_{\bar{\alpha}}\dfrac{\partial f}{\partial\bar{u}^\beta}\bigg)\bar{\mathcal{Z}}^\alpha,\quad
\bar{\partial}^vf &= (\dot{\partial}_{\bar{\alpha}}f)\delta\bar{\mathcal{V}}^\alpha = \dfrac{\partial f}{\partial\bar{u}^\alpha}\delta\bar{\mathcal{V}}^\alpha.
\end{align*}
In particular,
\begin{equation*}
d\mathcal{Z}^\alpha = -\dfrac{1}{2}C^{\:\alpha}_{\beta\gamma}\mathcal{Z}^\beta\wedge\mathcal{Z}^\gamma  - \dfrac{1}{2}C^{\:\alpha}_{\bar{\beta}\gamma}\bar{\mathcal{Z}}^\beta\wedge\mathcal{Z}^\gamma,\qquad d\mathcal{V}^\alpha = 0.
\end{equation*}

This formalism introduced here will be further used in a study of Laplace type operators on the holomorphic Lie algebroid $E$.

\subsection{K\"ahler Finsler algebroids}

In order to define the K\"ahler condition in the case of a Lie algebroid, we first need to define a metric structure on the prolongation $\mathcal{T}_{\mathbb{C}}E$  by
\begin{equation}
\label{metric}
\mathcal{G} = h_{\alpha\bar{\beta}}\mathcal{Z}^\alpha\otimes\bar{\mathcal{Z}}^\beta + h_{\alpha\bar{\beta}}\delta\mathcal{V}^\alpha\otimes\delta\bar{\mathcal{V}}^\beta.
\end{equation}

We now check if the Chern-Finsler connection \eqref{C-F} is metric with respect to the structure $\mathcal{G}$ defined above. First, note that, due to the fact that the components of the metric structure \eqref{metric} on $\mathcal{T}E$ depend only on $(z^k,u^\alpha)$, then the action of the vector fields $\{\mathcal{X}_\alpha,\mathcal{X}_{\bar{\alpha}},\mathcal{V}_\alpha,\mathcal{V}_{\bar{\alpha}}\}$ on the components of $\mathcal{G}$ is
\begin{align*}
&\mathcal{X}_\alpha h_{\beta\bar{\gamma}} = \delta_\alpha h_{\beta\bar{\gamma}},\quad
\mathcal{X}_{\bar{\alpha}} h_{\beta\bar{\gamma}} = \delta_{\bar{\alpha}} h_{\beta\bar{\gamma}},\\
&\mathcal{V}_\alpha h_{\beta\bar{\gamma}} = \dot{\partial}_\alpha h_{\beta\bar{\gamma}},\quad
\mathcal{V}_{\bar{\alpha}} h_{\beta\bar{\gamma}} = \dot{\partial}_\alpha h_{\beta\bar{\gamma}},
\end{align*}
such that the identity
\begin{equation*}
X\mathcal{G}(Y,Z) = \mathcal{G}(\mathcal{D}_XY,Z) + \mathcal{G}(Y,\mathcal{D}_XZ)
\end{equation*}
is easily checked as in the classical case, for $Y=\mathcal{V}_\beta,\ Z=\mathcal{V}_{\bar{\gamma}}$ and $X=\mathcal{X}_\alpha$ or $X=\mathcal{X}_{\bar{\alpha}}$ (\cite{15}). This means that the Chern-Finsler connection on the prolongation is \textit{metric} with respect to the structure $\mathcal{G}$. 

Let us now consider the horizontal $2$-form
\begin{equation} \label{f.K}
\Theta^h = -i h_{\alpha\bar{\beta}}\mathcal{Z}^\alpha\wedge\bar{\mathcal{Z}}^{\beta}.
\end{equation}
We can define a K\"ahler-Finsler algebroid following the idea of Aikou, \cite{2}.

\begin{defn}
A holomorphic Lie algebroid is called K\"ahler-Finsler algebroid if the horizontal K\"ahler form \eqref{f.K} is $h$-closed, i.e., $d^h\Theta^h=0$.
\end{defn}

The condition from the definition above easily yields
\begin{equation*}
\delta_\gamma h_{\alpha\bar{\beta}} = \delta_\alpha h_{\gamma\bar{\beta}},\quad \delta_{\bar{\gamma}} h_{\alpha\bar{\beta}} = \delta_{\bar{\beta}} h_{\alpha\bar{\gamma}},
\end{equation*}
which, due to \eqref{L,C}, becomes
\begin{equation}
\label{cond.K}
L^{\:\sigma}_{\alpha\gamma} = L^{\:\sigma}_{\gamma\alpha},
\end{equation}
a similar condition to the one from the classical case.

\end{document}